\documentclass{article}

\usepackage[utf8]{inputenc} % allow utf-8 input
\usepackage[T1]{fontenc}    % use 8-bit T1 fonts
\usepackage{hyperref}       % hyperlinks
\usepackage{url}            % simple URL typesetting
\usepackage{booktabs}       % professional-quality tables
\usepackage{amsfonts}       % blackboard math symbols
\usepackage{nicefrac}       % compact symbols for 1/2, etc.
\usepackage{microtype}      % microtypography

\usepackage[ruled,vlined,linesnumbered]{algorithm2e}

\SetAlgorithmName{Algorithm Schema}{\autoref}{List of Algorithms}

\SetKwFor{ParallelForEach}{for each}{in parallel do}{endfor}
\SetKwFor{ParallelCForEach}{for each computer}{in parallel do}{endfor}
\SetKwFor{ParallelTForEach}{for each thread}{in parallel do}{endfor}
\SetKwFor{WhileNotSatisfied}{while termination criteria are not satisfied}{}{}%

\usepackage{amssymb,amsmath,graphicx,xcolor,dsfont,parskip,calc,mathptmx,tabularx,pdflscape,arydshln}

\newcommand{\bmat}[1]{\begingroup%
	\begin{bmatrix} #1 \end{bmatrix}
	\endgroup}

\newcommand{\one}{\scalebox{1.1}{$\mathbf{1}$}}
\newcommand{\reals}{\ensuremath{\mathbb{R}}}
\newcommand{\integers}{\ensuremath{\mathbb{Z}}}
\newcommand{\coloneq}{\mathrel{\mathop:}=}
\newcommand{\eqcolon}{=\mathrel{\mathop:}}
\DeclareMathOperator{\tr}{tr}
\DeclareMathOperator{\card}{card}
\DeclareMathOperator{\supp}{supp}

\DeclareMathOperator{\trace}{tr}

\newcommand{\tld}[1]{\widetilde{#1}}

\newcommand{\skipthis}[1]{{}}

\newcommand{\Cadj}{C^{\textrm{a}}}
\newcommand{\Ccentr}{C^{\textrm{c}}}
\newcommand{\Cclosest}{C^{\textrm{d}}}

\usepackage{amsthm} 

\newtheorem{Thm}{Theorem}
\newtheorem*{Thm*}{Theorem}
\newtheorem{Lem}[Thm]{Lemma}
\newtheorem{Def}[Thm]{Definition}
\newtheorem{problem}[Thm]{Problem}

\usepackage{subfig}

\begin{document}
\title{The Use of Presence Data in Modelling Demand for Transportation}
 
\author{Jonathan Epperlein$^1$,
Jaroslaw Legierski$^{2,3}$,
Marcin Luckner$^3$, \\
Jakub Mare{\v c}ek$^1$\thanks{
Jakub can be reached at jakub.marecek@ie.ibm.com.
1: IBM Research -- Ireland, B3 IBM Campus Damastown, Dublin 15, Ireland.
2: Orange Labs Poland, 02-691 Warsaw, ul. Obrzezna 7, Poland.
3: Warsaw University of Technology, Faculty of Mathematics and Information Science, 
00-662 Warsaw, ul. Koszykowa 75, Poland.
}\; and
Rahul Nair$^1$}

\maketitle

\begin{abstract}
We consider the applicability of the data from operators of cellular systems to modelling demand for transportation. While individual-level data may contain precise paths of movement, stringent privacy rules prohibit their use without consent. Presence data aggregate the individual-level data to information on the numbers of transactions at each base transceiver station (BTS) per each time period. Our work is aimed at demonstrating value of such aggregate data for mobility management while maintaining privacy of users. In particular, given mobile subscriber activity aggregated to short time intervals for a zone, a convex optimisation problem estimates most likely transitions between zones. We demonstrate the method on presence data from Warsaw, Poland, and compare with official demand estimates obtained with classical econometric methods.

\end{abstract}

\section{Introduction}
\label{sec:intro}

Estimating travel demand is a key step of the transportation planning process. Cities typically build a modeling framework 
and use surveys along with observations to characterize how, where and by what means citizens move. These estimates are then 
used for management of existing infrastructure or for design of new services. Such an approach is resource-intensive, time 
consuming, and therefore can only be performed infrequently providing estimates that may be out-of-date. 

Alternatively, current movement patterns can be estimated based on data from mobile devices.
Several works have studied digital trajectories, or ``breadcrumbs'', which provide a very rich demand profile. 
Nevertheless, legislation and privacy regulation often restrict the use of such methods.
In Europe, in particular, requirements on processing telecommunications transmission data are set by European directives:
\begin{itemize}
\item Personal data should be protected, as per directive on e-Privacy(2002/58/EC) and directive on personal data protection (95/46/EC).
\item The consent of the data subject is required for processing any subscriber personal data.
\item The anonymized data  can be processed without consent only if they are aggregated  in the first step of processing and de-anonymization 
of the data is not possible. The anonymized record cannot not be connected with any subscriber in any case.
\end{itemize}
The use and processing of CDR records without the permission of the subscriber is hence not possible in the Europe,
and one may need to rely on aggregate data.

In this paper, we present methods to estimate trip generation and trip distribution rates from aggregate presence data. 
In particular, we consider mobile subscriber's location statistics aggregated in accordance with the above-mentioned regulations. 
We the present a linear program to be solved for each period of the discretisation of time to obtain the most likely estimate of movements
since the previous period.
We also present algorithms for obtaining the estimates of movements of likely movements over a number of periods,
or consider a distribution of trip durations.
Crucially, we show that the linear program can be solved in linear time, 
which makes the approach applicable in practice.

% JAKUB - could you perhaps provide a gentle introduction to the method here

\section{The Problem}

\subsection{Data}

The data comes in the form of time-stamped averages of event counts for each of a number 
of geographical zones: \textit{(Zone, Timestamp, Event Count)}, where the zones are
 defined by the local public transport operator.

More specifically: Say there are $N_Z$ zones $Z_1, \dotsc, Z_{N_Z}$, and we have a total of
 $N_t$ time intervals $[t_0, t_1],[t_1 t_2],\dotsc,[t_{N_t-1},t_{N_t}]$ for which the event 
 counts are reported; here, the intervals are all of equal length $T$, i.e.\ $t_\ell-t_{\ell-1}=T$, 
 but this is not necessary for the applicability of what follows. %  (make sure that is true once we're done)
Thus, the data consists of tuples of the form 
\[
	(Z_k, t_\ell, E_{k,\ell}) \qquad \text{ for } k=1,\dotsc,N_Z, \; \ell=1,\dotsc, N_t,
\]
which encode that ``In zone $Z_k$ during the time interval $[t_{\ell-1}, t_\ell]$, $E_{k,\ell}$ events were observed.''
 Note that the timestamp of the \textit{end} of the interval is reported in the dataset, and that there are $N_Z\cdot N_t$ such tuples.

For the purposes of this paper, we consider the number of events as a measure of how many unique terminals (phones, pagers, etc.) 
were present in the give zone over the given time interval. 

\subsection{Modelling user movement}

The given data suggests \textit{how many} terminals were present at any zone, but, due to privacy restrictions, not \textit{which}
 terminals were present where at what time. 
Hence, the information as to how users move from one zone to another is lost, and the problem we address in this paper is how to estimate 
this information from the aggregate data that we are given.
The problems treated are as follows:\\

\begin{problem}[One-Step Transitions]
    \label{prob1}
Given aggregate presence data for disjoint geographical zones at two points in time, 
and a cost function mapping each pair of geographical zones to a real number, 
estimate for each zone the proportion of users travelling  to each other zone,
within the time interval between the two points in time. 
\end{problem}

\begin{problem}[$k$-Step Transitions]
    \label{probK}
Given aggregate presence data for disjoint geographical zones at two points $t_1, t_2$ in time 
%\comment{This problem is about extrapolation, if you have one time-step, what do you estimate for $k$ timesteps, right? So I don't see what you need more than two timesteps for}, 
and a cost function mapping each pair of geographical zones to a real number, 
estimate for each zone the proportion of users  travelling to each other zone,
within time $k (t_2 - t_1)$. 
\end{problem}

\begin{problem}[Duration-$L$ Transitions]
    \label{prob2}
Given aggregate presence data for disjoint geographical zones, for two or more points $t_1, t_2, \ldots$ in time, 
and a cost function mapping each pair of geographical zones to a real number, 
and a given trip duration $L$, 
estimate for each zone the proportion of users travelling to each other zone,
with trip duration being $L$. 
\end{problem}

\begin{problem}[Realistic Transitions]
    \label{prob3}
Given aggregate presence data for disjoint geographical zones, for two or more points $t_1, t_2, \ldots$ in time, 
and a cost function mapping each pair of geographical zones to a real number, 
and a distribution of trip durations, as a histogram approximation of a probability density function $f_L$, 
estimate for each zone the proportion of users travelling to each other zone.
\end{problem}

Absent further information, we have to make assumptions about user behaviour, and the main assumption is that the users 
redistribute themselves in an \textit{optimal} way 
-- optimal with respect to a cost function, the design of which requires insights into the underlying transportation structure.
This is best explained via a trivial-seeming example:\\
Assume first, that there are only two zones, $Z_1$ and $Z_2$, and that we have the following data:
\[
(Z_1,t_1,3),\; (Z_2,t_1,1),\; (Z_1,t_2,2) ,\; (Z_2,t_2,2).
\]
Intuitively, it seems obvious to expect that one user went from Zone 1 to Zone 2, whereas all other stayed put. 
Of course, it is also possible that two users went from Zone 1 to Zone 2 and the sole user in Zone 2 went to Zone 1. 
However, if we associate a \textit{cost} to users moving, then the first, intuitive solution, is clearly the optimal one.

%This can be encoded in a flow matrix $ = \{ t_{o,d} \}$:
%\begin{align}
%    t_{o,d} = \prob(\textrm{arrival at destination } d \textrm{ in } 1 \textrm{ time-step} | \textrm{starting from origin } o )
%\end{align}
%whose relationship to the usual origin-destination (OD) matrices used in transportation engineering we will
%clarify shortly.

Assume now, that there is a third zone, $Z_3$, and that we additionally have the following data:
\[
	(Z_3, t_1, 1), \; (Z_3, t_2, 1).
\]
A natural assumption is now that the user in Zone 3 simply stayed there. But what if Zone 3 is connected to Zones 1 
and 2 through rapid transit, whereas the only convenient way of getting from Zone 1 to Zone 2 is through Zone 3, and this 
journey takes longer than $t_2-t_1$?  In this case, the cost of travel from $Z_1$ to $Z_2$ would be quite high, and the 
most reasonable and optimal solution would be to expect one user to move from $Z_1$ to $Z_3$ and one user to move from $Z_3$ to $Z_2$.

This example illustrates the idea of optimal user movement and highlights, how knowledge of the transportation system needs
 to be incorporated in the design of the cost function. 
One can, for instance, benefit from the estimates of the travel-times provided by Google Maps \cite{Wang2011}, 
 use the Euclidean distance, or a discrete metric, in which adjacent pairs of zones have cost 0 and all other pairs have cost 1. 
Either way, we will solve that problem by casting it as a linear program, which will be formally defined in the next section.

Finally, in Problem \ref{prob3}, note that we treat the travel time of all users of the road network as a
 random variable $L$. 
For this random variable $L$, one assumes there exists probability density function $f_X$, i.e., a non-negative 
%Lebesgue-integrable 
function:
\begin{align}
\Pr[a\leq L\leq b]=\int _{a}^{b}f_{L}(x)\,dx.
\end{align}
On the input, we assume a histogram approximation thereof, with the bins of the histogram centered at multiples
of the interval at which we sample the input data, i.e., $t_2 - t_1 = t_{k+1} - t_{k}$.
For example for the sampling at 1 minute, we would be given a list:
\begin{align}
\Pr[0 \leq L < 1.5 ], \Pr[1.5 \leq L < 2.5 ], \ldots
\end{align}
We note that such distributional data are regularly obtained by most operators of public transport, 
either by surveys
or by studying traces of individuals using personal seasonal tickets,
although the alignment of the bins may not readily match the sampling interval of the presence data.

\section{The Linear Program}

First, we show that Problem \ref{prob1} can be solved by a linear program. 

\subsection{Notation and preliminaries}

Throughout, $\integers_+$ denotes the nonnegative integers,
$n$ and $k$ are positive integers, $\one_{n}$ denotes an $n$-dimensional column vector of all $1$s, and $I_n$ is the $n$-dimensional identity matrix; the dimension $n$ is omitted if it is clear from the context.
We denote matrices by capital letters, and their elements by the corresponding lower-case letters, e.g.\ for a matrix $M$, $m_{ij}$ denotes the element in the $i^\text{th}$ row, $j^\text{th}$ column; $M^T$ denotes the transpose of $M$; $\tr(M)=\sum_i m_{ii}$ denotes the trace of the square matrix $M$.
Note that with this notation, we obtain the vector $\eta$ of row-sums of $M$, i.e.\ $\eta$ with $\eta_i = \sum_j m_{ij}$, as $\eta=M\one$ and analogously we have $\gamma^T = \one^T M$ for the vector of column-sums.

To keep the notation compact, we consider user movement between two consecutive time intervals only; else, we'd have to include an index corresponding to the time stamp in everything, which is unnecessary for the developments. In other words, we consider a data set of the form
\begin{align}
	(Z_1,t_1, E_{11}),\; & (Z_2,t_1,E_{21}),\dotsc, \\
	& (Z_{N_z},t_1,E_{N_z 1}) ,\;(Z_1,t_2, E_{12}),\dotsc, (Z_{N_z},t_2,E_{N_z 2}). \notag
\end{align}

Let $E_\ell$ denote the (column) vector of users present in all zones at time $t_\ell$, i.e.\ $E_1 = \bmat{E_{11}&E_{21}&\dotsm&E_{N_Z 1}}^T$ and so on.
The number of users moving from $Z_i$ to $Z_j$ is denoted by $x_{ij}$ and the matrix of flows is $X\in\reals^{N_z \times N_z}$; 
this is the quantity we are interested in finding. We denote the costs of moving from $Z_i$ to $Z_j$ by $c_{ij}\in\reals$ and collect 
them in the matrix $C\in\reals^{N_z \times N_z}$.

\subsection{Cost functions}
\label{costfunctions} 

The costs of moving from $Z_i$ to $Z_j$ is, in some sense, a parameter to the model.
We have conducted our experiments with three natural choices of the cost matrix $C\in\reals^{N_z \times N_z}$,
but we do not claim these are original.

Let us have each zone $Z_i$ associated with a polygon, which is defined by a sequence of corner-points in the Euclidean plane.
Let us denote $k_i$ corner-points associated with $Z_i$ by $(P_{i,1}^x, z_{i,1}^y), (P_{i,2}^x, P_{i,2}^y), \ldots, (P_{i,k_i}^x, P_{i,k_i}^y)$.
In the adjacency metric, we consider:
\begin{align}
    \Cadj(Z_i, Z_j) := \begin{cases}
    \; 0 & \quad \textrm{if} i = j \\
    \; 0.1 & \quad \textrm{if } \exists  1 \le k_1 \le k_{i}, 1 \le k_2 \le k_j \\
        & \textrm{ s.t. } \exists (P_{i,k_1}^x, P_{i,k_1}^y) = (P_{j,k_2}^x, P_{j,k_2}^y) \\
    \; 1 & \quad \textrm{otherwise} 
    \end{cases} 
\end{align}
Considering the centroid of $Z_i$:
\begin{align}
  \textrm{centroid}(Z_i) := \left( \frac{\sum_{j=1}^{k_i} P_{i,j}^x}{k_i} ,  \frac{\sum_{j=1}^{k_i} P_{i,j}^y}{k_i} \right)
\end{align}
The centroidal distance $\Ccentr(Z_i, Z_j)$ is the Euclidean distance between centroids 
 $\textrm{centroid}(Z_i)$ and $\textrm{centroid}(Z_j)$
of the two zones.
Finally, considering the closest corner-point associated with the other zone:
\begin{align}
  \textrm{NN}(Z_i, Z_j) & := \arg \min_{(P_{i,k_1}^x, P_{i,k_1}^y) 1 \le k_1 \le k_i } \min_{ (P_{j,k_2}^x, P_{j,k_2}^y) 1 \le k_2 \le k_i } 
  D(P_{i,k_1}, P_{j,k_2}) \notag \\
  D(A, B) & := \sqrt{ 
  (A^x - B^x)^2
  + (A^y - B^y)^2
  }. 
\end{align}
Distance $\Cclosest(Z_i, Z_j)$ is then the Euclidean distance between $\textrm{NN}(Z_i, Z_j)$ and $\textrm{NN}(Z_j, Z_i)$. 
One can clearly define many other cost matrices, e.g., considering free-flow travel times, and introducing randomisation. 

\subsection{Constraints assuming a constant number of users over time}
\label{constantnousers} 

As a first step, assume that $\sum_{i=1}^{N_z} E_{i1} = \sum_{i=1}^{N_z} E_{i2} = N$, i.e.\ that the number $N$ of users in the entire network is the same in both time intervals. In order for a flow $X$ to explain the observations, it must ``remove'' $E_{i1}$ users from, and ``place'' $E_{j2}$ users back in each zone $Z_i$ (note that the number of users staying in $Z_i$ is $x_{ii}$). The number of users moving \textit{from} $Z_i$ is then given by the row-sum $\sum_j x_{ij}$, whereas the number of users moving \textit{to} $Z_i$ is given by the column-sum $\sum_k x_{ki}$. Of course, flows have to be nonnegative, too. Using the above notations, this translates to 
\begin{align}
\label{constantnousers1}
	X\one & = E_1 \\
	X^T \one & = E_2\\
	X & \geq 0. \label{constantnousers3}
\end{align}
We note here that this set of constraints defines what is known as the \emph{$(N_z,N_z)$-transportation polytope with 
marginals $E_1$ and $E_2$.} A further constraint can be added: in reality, only integer numbers of users can move, 
hence we could also require $X\in\integers^{N_z\times N_z}$. Since we are only approximating anyway, this might or 
might not be a good idea. Note also Lemma~\ref{lem:TP}, which suggests that for integer marginals, the 
minimizer is also integral.

The total cost associated with a flow $X$ is $\sum_{i=1}^{N_z}\sum_{j=1}^{N_z} c_{ij} x_{ij} = \tr(C^T X)$, and we find the flow $X$ minimizing the overall cost -- and hence an estimate of the true movement of users between zones -- as the solution to the following optimization problem:
\begin{equation}\label{eq:LP1}\tag{LP-C}
	\begin{aligned}
	\text{minimize } &\qquad& & z  = \trace(C^T X) \\
	\text{subject to } &&& %\left\{
	 \begin{aligned}
	X &\geq0 \\% & x_{ij}&\in\integers\\
	X\one &= E_1 & X^T \one &= E_2
	.
	\end{aligned}% \right.
	\end{aligned}
\end{equation}

%The polytope has been studied extensively \cite{schrijver2002,DeLoeraKim2014} over the past 70 years. 
%Where convenient, we will rewrite~\eqref{eq:TP} more compactly as 
%$X^T \one_p = \eta, X \one_q = \gamma$.

\subsection{Constraints allowing for time-varying numbers of users}

We cannot assume that the number of users stays constant: Users are arriving in the covered area and leaving it, and of 
course devices are being turned off and on. To address that, we introduce a  sink/source zone and index it by $N_z+1$. For this zone we have
%source and a sink and index them by $N_z+1$ and $N_z+2$, respectively. For those two zones, we have
\begin{align*}
	c_{i,N_z+1} & = \text{cost of user disappearing from $Z_i$} \\
	c_{N_z+1,i} & = \text{cost of user spawning in $Z_i$}
%	c_{i,N_z+2} & = \text{cost of user disappearing from $Z_i$} & c_{N_z+2,i} & = \infty.\\
\end{align*}

The number of users in the sink/source zone can only be computed after the data at $t_2$ is available. Then, we let 
\[
	\delta_{12} = \sum_{i=1}^{N_z} E_{i1} - \sum_{i=1}^{N_z} E_{i2}	
\]
and
\[
	E_{N_z+1,1} = u, \qquad E_{N_z+1,2} = u + \delta_{12},
\]
where selecting a parameter $u>0$ allows for consideration of users disappearing from $Z_i$ and spawning in $Z_j$ instead 
of travelling from $Z_i$ to $Z_j$. One might also consider adding another sink/source for this specific purpose to gain more 
control over the cost: for instance, if $c_{i,N_z+1}$ and $c_{N_z+1,j}$ are too small compared to $c_{ij}$, the optimization 
problem will always prefer disappearing/spawning over travel.

\section{The Matrix Manipulations}

Second, one should like go beyond a snapshot of transitions within a single interval between two points in time, at which the presence data are available.
We show that one can manipulate the solution of the linear program (LP) of the previous section
to derive first the $k$-step transitions of Problem \ref{probK} and subsequently the solutions to 
Problems \ref{prob2} and \ref{prob3} by straightforward matrix manipulation.

\subsection{Extrapolating the most recent one-step transitions} 
\label{matrixpowers}
%TODO: Raise to some power.

One possible solution to Problem \ref{probK} involves considering the output of the LP
 as a matrix and raising it to the appropriate power.
Let us return to our example, where we had
\[
(Z_1,t_1,3),\; (Z_2,t_1,1),\; (Z_1,t_2,2) ,\; (Z_2,t_2,2).
\]
We expect the flow of users to be $X = \left[\begin{smallmatrix}
2 & 1 \\ 0 & 1
\end{smallmatrix}\right]$, in other words two users remain in $Z_1$, one goes from $Z_1$ to $Z_2$, and the user already in $Z_2$ remains there. Another interpretation is in terms of proportions: One third of the users in $Z_1$ move to $Z_2$, whereas two thirds stay put, and so on. Mathematically, that corresponds to normalizing each column of $X$ to sum up to one (i.e.\ making $X$ into a row-stochastic matrix), say  $\tld{X}= \left[\begin{smallmatrix} \frac{2}{3} & \frac{1}{3} \\ 0 & 1 \end{smallmatrix}\right]$.

If this trend were to continue, then we'd have for the number of users in $Z_1$ and $Z_2$ at time $t_3$:
\begin{align*}
E_{13} & = \tld{x}_{11} E_{12} + \tld{x}_{21} E_{22} & \text{ and }& &
E_{23} & = \tld{x}_{12} E_{12} + \tld{x}_{22} E_{22}.
\end{align*}

This can be written compactly, and more generally, as the matrix-vector product
$
	E_{\ell+1} = \tld{X}^T E_\ell,
$
and with $E_1$ given, we have $E_2=\tld{X}^TE_1$ and $E_3=\tld{X}^TE_2=\tld{X}^T\tld{X}^TE_1$ and eventually we get 
\[
	E_{\ell+1} = \left(\tld{X}^\ell\right)^T E_1.
\]

In the context of probability distributions (where the elements $E_{k\ell}$ denote probability mass instead of users), this is a well-known and thoroughly researched Markov chain, and all results apply here. The most important one is that, should the trend continue forever, the system will (under mild conditions) settle into a steady state. More specifically, there is a distribution of users $f$ such that $\tld{X}^Tf=f$ and as $\ell\rightarrow\infty$, we have $E_\ell\rightarrow f$.

%It seems obvious that if this pattern of movement were to continue, 
%all the activity would move to the second zone, eventually:
%
%\[
%(Z_1,t_4,0) ,\; (Z_2,t_0,4).
%\]
%
%Should this be realistic, it is clearly possible to raise the one-step
%transition matrix $T$
%to a suitable power $k$ to obtain the $k$-step transitions $T^{(k)} = 
%\{ t_{o,d}^{(k)} \}$:
%
%\begin{align}
%    t_{o,d}^{(k)} = \prob(\textrm{arrive at destination } d \textrm{ in } k \textrm{ time-steps } | \textrm{ starting from origin } o )
%\end{align}
%
%and from the basics of the theory of Markov chains \cite{Shorten2015},
% it is clear that $t_{i,j}^{(k)}$ is the sum of
% probabilities of travel along all paths from $i$ to $j$ along paths of length $k$ steps.
%The related Chapman-Kolmogorov equation has a solution
%\begin{align}
%    T^{(k)} = \begin{cases}
%    T^k & n \ge 2 \\
%    T   & \textrm{otherwise}
%    \end{cases}
%\end{align}
%  where $T^k$ is the usual matrix power. 

\subsection{Extrapolating $k$ recent one-step transitions}

If instead of just event counts at two time steps, we have event counts at several consecutive time steps, 
another solution to Problem \ref{probK} would be to multiply by the estimated one-step flow matrices in turn. To return to our example, 
let us extend it to:
\[
(Z_1,t_1,3),\; (Z_2,t_1,1),\; (Z_1,t_2,2) ,\; (Z_2,t_2,2),\; (Z_1,t_3,3),\; (Z_2,t_3,1).
\]
The flows between times $t_1$ and $t_2$ would be as before, and from $t_2$ to $t_3$ we'd have $X_2 = \left[\begin{smallmatrix}
2 & 0 \\ 1 & 1
\end{smallmatrix}\right]$, or after normalization:
\begin{align*}
	\tld{X}_1 &= \bmat{\frac{2}{3} & \frac{1}{3} \\ 0 & 1} & \text{and}&& 
	\tld{X}_2 &= \bmat{1 & 0 \\ \frac{1}{2} & \frac{1}{2}}.
\end{align*}

Now the estimate for $E_4$ would be $E_4 = \tld{X}_1^TE_3 = \tld{X}_1^T\tld{X}_2^TE_2=\tld{X}_1^T\tld{X}_2^T\tld{X}_1^TE_1$ and in general
\[
E_{\ell+1} = 
	\begin{cases}
		\left(\tld{X}_1^T \tld{X}_2^T\right)^{\ell/2} E_1 & \text{if } \ell \text{ even}\\
		\left(\tld{X}_1^T \tld{X}_2^T\right)^{(\ell-1)/2} \tld{X}_1^T E_1 & \text{if } \ell \text{ odd.}
	\end{cases}
\]

%It seems obvious that if this pattern of movement were to continue, 
%the small fluctuations would continue.
%\comment{this is not entirely true, because the 
Because the product of row-stochastic matrices is again row-stochastic, this new Markov chain
may have a stationary distribution appearing every $k$ time steps.

%Should this be realistic, it is again possible to obtain the $k$-step 
%transitions $T^{(k)} = \{ t_{o,d}^{(k)} \}$ by matrix multiplication. 
The apparent difficulty with a naive implementation of these approaches is runtime, in particular when $k$ is large.
As we will show in the next section, there are surprisingly efficient algorithms, though.

\subsection{Approximating duration-$L$ transitions}

%\begin{problem}[Duration-$L$ Transitions]
%    \label{prob2}
%Given aggregate presence data for disjoint geographical zones, for two or more points $t_1, t_2, \ldots$ in time, 
%and a cost function mapping each pair of geographical zones to a real number, 
%and a given trip duration $L$, 
%estimate for each zone the proportion of users travelling travelling to each other zone,
%with trip duration being $L$. 
%\end{problem}

In order to solve Problem \ref{prob2}, it suffices to consider a convex combination of 2 solutions to Problem \ref{probK}
for suitable $k$.
Consider the computation of duration-$L$ transitions, where there exist presence data at times $t_{k}, t_{k+1}$ with $t_{k} \le L \le t_{k+1}$.
Consider the $k$-step transition matrix $S_k$ and $(k+1)$-step transition matrix $S_{k+1}$ computed as a solution to Problem \ref{probK}. Clearly,
\begin{align}
S = \frac{L - t_k}{t_{k + 1} - t_k} S_k + \frac{t_{k+1} - L}{t_{k + 1} - t_k} S_{k+1} 
\end{align}
is the solution to Problem \ref{prob2}.

%Read till here. (Jonathan)
%\hrulefill

\subsection{Approximating realistic transitions}
\label{realistictransitions}

%\begin{problem}[Realistic Transitions]
%    \label{prob3}
%Given aggregate presence data for disjoint geographical zones, for two or more points $t_1, t_2, \ldots$ in time, 
%and a cost function mapping each pair of geographical zones to a real number, 
%and a distribution of trip durations, as a histogram approximation of a probability density function $f_L$, 
%estimate for each zone the proportion of users travelling travelling to each other zone.
%\end{problem}

Finally, in order to solve Problem \ref{prob3}, one has to notice that the
solution is a convex combination of $H$ solutions to Problem \ref{probK}
for $H$ bins of the histogram approximation.
Consider the $H$ bins of the histogram with values: 
\begin{align}
h_1 & = \Pr[0 \leq L < 1.5 ] \notag \\
h_2 & = \Pr[1.5 \leq L < 2.5 ] \notag  \\
& \ldots  \notag  \\ 
h_H & = \Pr[\max{\supp{L}} - 1 \leq L < \max{\supp{L}} ], \notag 
\end{align}
where $\max{\supp{L}}$ is the largest possible realisation of $L$.
It is clear that $h_i, 1 \le i \le H$ can be seen as a probability mass function on a discretisation of L,
with $\sum_{i = 1}^{H} h_i = 1$.
One can use $h_i$ as weights in a convex combination of the transition matrices, which is a sum.
For the first bin, we consider a solution $S_1$ of Problem \ref{prob1} with weight $h_1$ as the first summand.
For the second bin, we consider a 2-step transition matrix $S_2$ computed as a solution to Problem \ref{probK} with weight $h_2$ as the second summand.
Subsequently, we consider $k$-step transition matrix $S_k$, $2 < k \le H$ computed as a solution to Problem \ref{probK} with weight $h_k$
to obtain:
\begin{align}
S = \sum_{i = 1}^{H} (h_i S_i),
\end{align}
which is the solution to Problem \ref{prob3}.

\section{Run-Time Analysis}

\begin{table}[t]
\begin{tabular}{l|l|l}
    \toprule 
    Method                                & Run-time                    & Ref. \\
    \midrule
    General-purpose interior-point method & $O( n^7 \ln(1/\epsilon) ) $ & \cite{gondzio2012} \\
    Hungarian method                      & $O( n^4 )$                  & \cite{kuhn1955}  \\
    Augmenting-path algorithm             & $O( n^3 )$                  & \cite{jonker1987}  \\
    \midrule
    The proposed algorithm            & $O(n)$                    & Theorem \ref{thm:tracemax} \\
    \bottomrule
\end{tabular}
\caption{The complexity of solving the linear program \eqref{eq:LP1} in dimension $n \times n$.
We note that it makes it possible to obtain the objective function value at cost $O(n)$, while 
retrieving the $n \times n$ matrix, at which this value is attained requires time $O(n^2)$. }
\end{table}

\subsection{Linear Programming}

Next, let us consider whether a faster algorithm for solving the class of linear programming problems \eqref{eq:TP} are possible. We make use of:

\begin{Lem}	[E.g.\ {\cite[Lemma~2.2 and Corollary~2.11]{DeLoeraKim2014}}]\label{lem:TP}
For the $(p,q)$ transportation polytope defined by $\gamma$ and $\eta$ we have:
	\begin{enumerate}%[topsep=-3ex]
		\item it is nonempty if and only if $\one^T_q \eta=\one^T_p \gamma$, in other words if the sums of the marginals are equal;
		\item all its vertices are integral, if $\eta\in\integers^q$ and $\gamma\in\integers^p$, i.e.\ if the marginals are integral.
	\end{enumerate}
\end{Lem}

but we stress that the polytope has been studied very extensively over the past 70 years and refer to \cite{rachev1998mass,rachev1998mass2,schrijver2002} as
standard references.

Now we state the main result:

%\begin{Thm}\label{thm:tracemaxGeneral} Consider  
%$\gamma \in\integers_+^{n_1}$ and 
%$\eta \in\integers_+^{n_1}$, with 
%$\one^T_{n_1}\gamma= \one^T_{n_2}\eta = k$, 
%matrices $W \in \reals^{\Red{n_2 \times n_1}}$, $C\in\reals^{\Red{n_2\times n_2}}$ with weights $w_{i,i} \ge 0, c_{i,i}$ on the diagonal \Red{(but if $n_1\neq n_2$, what is the diagonal?)} and 0 elsewhere,
%and a non-decreasing function $f$, consider 
%\begin{equation}\label{eq:IP1}\tag{IP1}
%  z' := \max_{X \in\integers_+^{n_1 \times n_2}} \; f( \tr(W X + C) )  \;
%    \text{ s.t. } X \one_{n_1} = \gamma, X^T \one_{n_2} = \eta.
%\end{equation}
%%The maximum $z'$ as well as one $X$ where achieving it 
%are computable in time linear in $n$.
%\end{Thm}
% The proof is constructive and detailed in the appendix, only due to its length compounded by the requirements to prove the number of non-zeros in $X$ is linear, 
%and complications resulting from the non-square matrices ($n_1 \neq n_2$). Here, we prove a simpler version:
 
\begin{Thm}\label{thm:tracemax} Given $\eta, \gamma\in\integers_+^n$, with $\one^T\eta=\one^T\gamma=k$, consider
\begin{equation}\label{eq:ILP}\tag{IP2}
    z := \max_{X \in\integers_+^{n\times n}} \; \tr(X)  \;
    \text{ s.t. } X \one = \gamma, X^T \one = \eta.
\end{equation}
Then, $z = \sum_{i=1}^{n} \min\bigl\{	\eta_i,\gamma_i \bigr\}$.
This is computable in time linear in $n$.
\end{Thm}

For the proof, please see the Appendix. % \ref{proof:tracemax}.

\subsection{Sparse Matrix Multiplication}

\begin{table}[t]
\begin{tabular}{l|l|l}
    \toprule 
    Method                                & Run-time                    & Ref. \\
    \midrule
    Standard dense matrix multiplication     & $O( n^3 )$                  & Standard \\
    Present-best dense matrix multiplication & $O( n^{2.3728639} )$        & \cite{LeGall2014}  \\
    Fourier-transform-based methods          & $\tilde O( k + nb )$         & \cite{Pagh2013}  \\
    \bottomrule
\end{tabular}
\caption{The complexity of multiplying two sparse matrices, each in dimension $n \times n$,
each with $k$ non-zero coefficients, where the product has $b$ non-zero coefficients. }
\end{table}

In addition to the complexity of the solving of the instance of the linear programming problem, 
we may need to solve a number of 
matrix-matrix multiplication problems, where the two matrices are sparse, 
each in dimension $n \times n$,
each with $k$ non-zero coefficients, 
and where the product has $b$ non-zero coefficients. 
Using traditional dense linear algebra, one can obtain the result in time $O(n^3)$, 
using more sophisticated methods for dense matrices, one can improve this to $O(n^\omega), \omega \approx 2.37$.

Considering the sparsity of matrices $A, B$, however, Pagh \cite{Pagh2013} has shown methods that there 
are based on conversion to polynomial:
$$p(x) = \sum_{k=1}^n (\sum_{i=1}^n A_{ik} s_1(i) x^{h_1(i)}) (\sum_{j=1}^n B_{kj} s_2(j) x^{h_2(j)}) $$
and utilisation of fast Fourier transform for polynomial multiplication,
which can compute $c_0,...,c_{b-1}$ such that $\sum_i c_i x^i = (p(x) \mod x^b) + (p(x) \div x^b)$ in time $\~O(n^2+ n b)$,
with ${(AB)}_{ij} = s_1(i) s_2(j) c_{(h_1(i)+h_2(j))} \mod b$.
%An unbiased estimator of $(AB)_{ij}$ with variance at most $||AB||_F^2 / b$, where $||AB||_F$ denotes the Frobenius norm of $A$,
% can then be computed as:
%$C_{ij} = s_1(i) s_2(j) c_{(h_1(i)+h_2(j))} \mod b$
Eventually, Pagh \cite{Pagh2013} shows that the algorithm computes $AB$ exactly in time $\~O(N + nb)$ with high probability.
Although these algorithms may not be easy to implement,  
related algorithms with simpler hashing functions are readily implementable.

\section{Computational Experiments}
\label{experiments}

%\maketitle              % typeset the title of the contribution

%
%\subsection{Data description}

%An important factor of this procedure is a protection of a user's privacy. The data are never 
%collected in the other form that statistics describing a load of a base transceiver station (BTS).

\begin{table}[t!]
\caption{Types of registered events: First two types are generated by a user (``active''), while the following eight are generated by the network
 (``passive'').}\label{tab:events}
%\begin{tabular}{|p{2cm}|p{10.5cm}|}
%\hline
\begin{itemize}
\item A cell identifier change reported in the access request procedure of any mobile originated (MO) or mobile terminating (MT) transaction % of the mobile station (MS) 
\item A cell identifier change reported when a cell identifier different from the one stored in the visitor location register (VLR)
 is used in response to any mobile terminating (MT) transaction % when in the page response, the cell identifier
\end{itemize}
% \hline
\begin{itemize}
\item  Location update performed by a mobile station
\item  International mobile subscriber identity (IMSI) attached to a mobile station changes
\item  A subscriber is deleted from the visitor location register 
\item  The subscriber switches off the mobile station
\item  Mobile station is considered as detached by the network after a long period of inactivity
\item  Cell identifier change is detected during GPRS activity
\item  Cell identifier change is detected during the provision of subscriber information %procedure for a GPRS subscriber.
\item  During long calls, the visitor location register (VLR) may receive request messages to refresh subscriber data.
In this way, the VLR notices that the cell identifier arriving in the message differs from the one stored in a database.
 \end{itemize}
% \hline
%\end{tabular}
\end{table}

For our preliminary experiments, we have used data collected from the Public Land Mobile Network (PLMN) of Orange Polska 
within the municipal area of Warsaw. 
For each base transceiver station (BTS) of each separate cellular system (2G/3G),
  we have received the total numbers of connections from unique terminals per a unit of time. 
No data other than the aggregate statistics were collected.
  
In particular, the data were collected by an internal transmission system, in which some types of network events are associated with a location.  
Such events can be subdivided into events generated by the user (``active'') or generated by the network (``passive''). 
In Table~\ref{tab:events}, the first two bullet points capture events generated by the user. 
These events relate to mobile terminating (MT) transactions, i.e., 
events registered during an active use of a mobile terminal by the user. 
The following bullet points list event types generated by the network. These events are generated without an active participation of the user
 and include operations of the so-called visitor location register (VLR), which tracks subscriber roaming within 
a mobile switching centre (MSC) location area, cell identifier changes, International Mobile Subscriber Identity (IMSI) changes, 
location updates performed by a mobile station, and the switching of mobile stations.

\begin{figure}[t]
\centering
\includegraphics[width=\columnwidth]{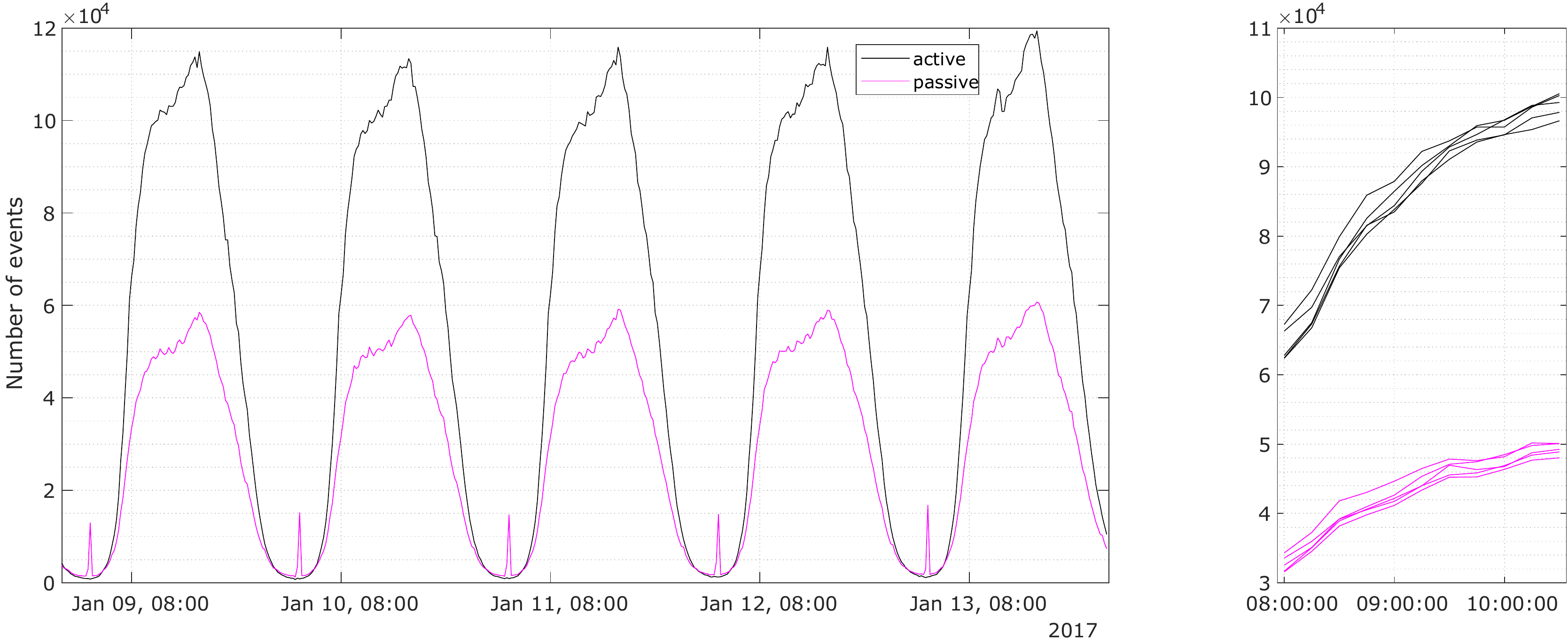}
\caption{Total number of events vs time stamps, split by passive (magenta) and active (black) events. The spike in passive events at 3:15am is an artefact of the data acquisition system's operation. 
}
\label{fig:events}
\end{figure}

\begin{figure}[t]
%\subfloat[active]{\includegraphics[width= 0.45\textwidth]{Pix2/active.pdf}}
%\subfloat[passive]{\includegraphics[width= 0.45\textwidth]{Pix2/passive.pdf}}
\includegraphics[width=\columnwidth]{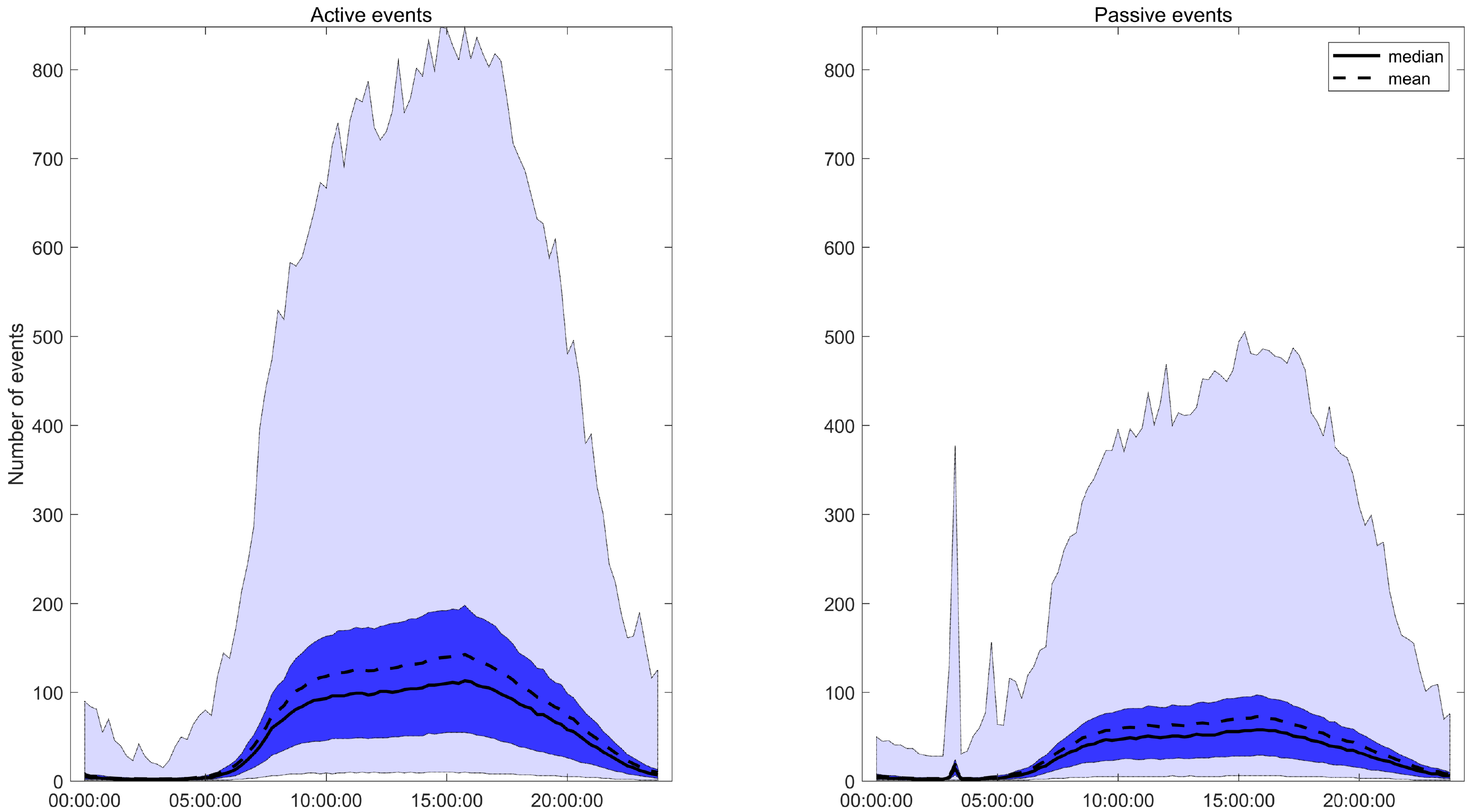}
\caption{The evolution of some statistics of the active (left) and passive (right) events observed. The statistics are computed over all BTSs and all 5 available days (Monday -- Friday), e.g.\ the mean at 9:00 is the mean of the number of events at all BTSs on all 5 days at 9am. The inner (dark blue) band is the interquartile range, whereas the outer (light blue) band extends between the 5\textsuperscript{th} percentile and the maximum event count at that time.}
\label{fig:stats}
\end{figure}

Within the one week of data collected, %without an observation each type of the events separately. 
the most frequent events were location updates, while the cell identifier change was very rare.
Figure~\ref{fig:events} presents the evolution of the total number of events during the period. 
The volumes of both active and passive events are correlated with the time of 
day; the cyclical nature of the evolution of the volumes of events over time is clear. 
The evolution of the numbers of passive events is more uneven than that of the active events. 
The local maxima of the numbers of passive events at 3:15 a.\ m.\ are related with the operations of the data acquisition and processing system.
%For the passive events, we can observe peaks generated by the system at  3:15 every day.
Further, note that the observed events are not equally distributed in the city. 
Figure~\ref{fig:stats} presents statistics of the number of events observed across all BTSs at a given time. 
The large difference between mean and median, the wide interquartile range, and the very large difference between
 maximum and mean number of events at all times suggest a large variability in the distribution of events over the BTSs. 
This disproportion is the result of differing population density, 
 differing ranges of base transceiver stations, and  
 differing telecommunication technologies.

%For a single day, we have analysed the share of each event in the total number of events. 
%The shares are presented in Table~\ref{tab:events}. 
%When the events generated by active subscribers occurs in a similar proportion, the distribution
% of automatically generated events is uneven. 
%The most frequent event is a localisation update, when the cell ID change is very rare.

% TODO: Desribe the rest of the pipeline ***

For comparison, we have used the estimates presently employed by the public transport operator in Warsaw (ZTM).
One should like to point out that the estimates of ZTM
 imply a rather different distribution of the population across the zones than the density of events within the Orange network.
See Figure \ref{fig:spatial} for the density of events within the Orange network for three representative time periods.
In Figure~\ref{fig:presence}, we compare it to the implied presence data within the ZTM estimate: with each data point represents 
one zone and the two axes correspond to the presence implied by ZTM for the morning peak 
 and the numbers of events within Orange network during the same time period, averaged over the 5 work days. 
As can be seen, %on the right, where we plot a measure of the error between the estimates of the public transport operator
 %and numbers of events,
 our model starts with rather different data, and any direct comparison of our estimates of trip distributions
 with the estimates of the public transport operator (ZTM) will hence be imperfect.
 
\begin{figure*}[t!]
	\centering
	\subfloat[6 a.\ m.\ ]{\includegraphics[clip, trim={2cm 2cm 2cm 2cm}, width= 0.33\textwidth]{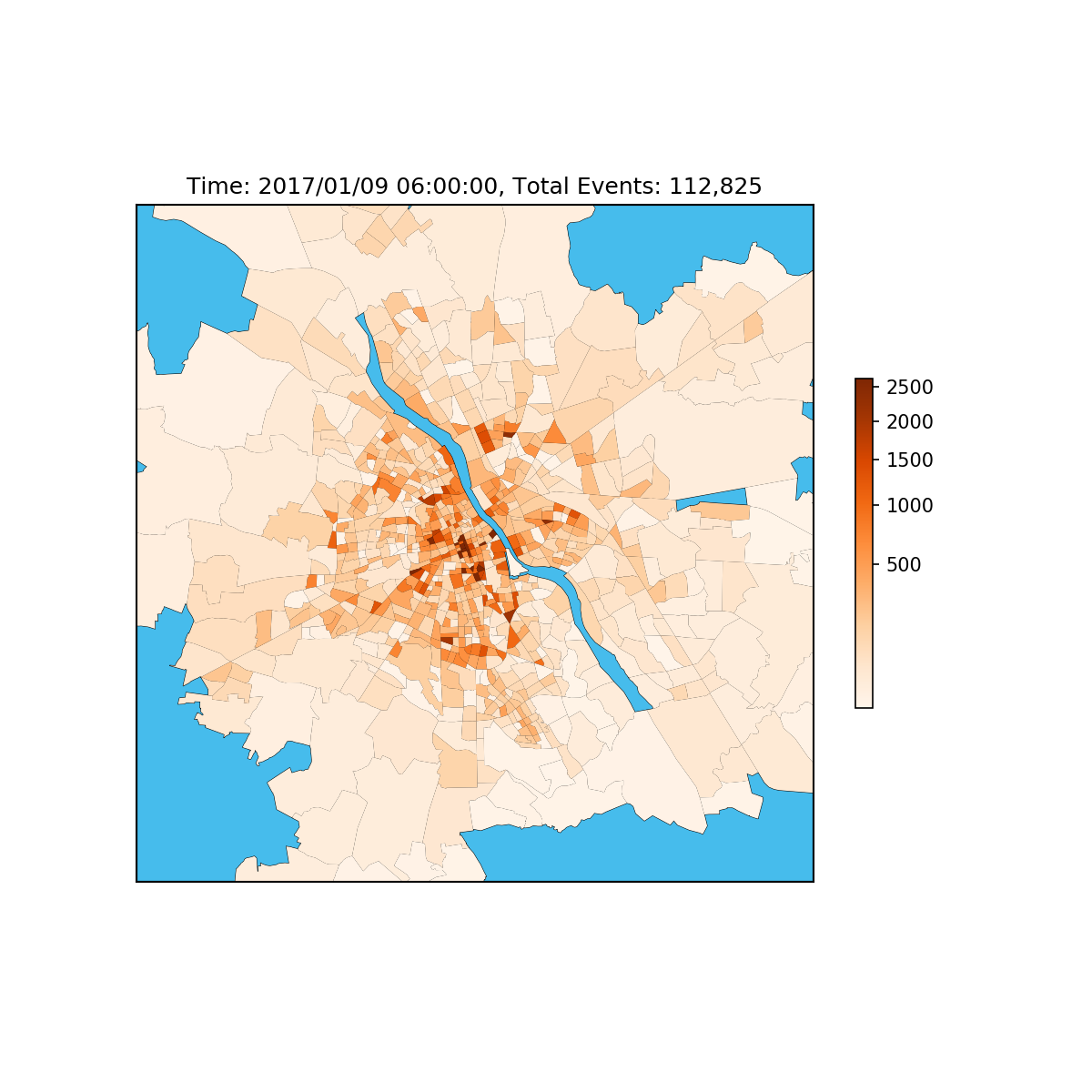}}
	\subfloat[9:30 a.\ m.\ ]{\includegraphics[clip, trim={2cm 2cm 2cm 2cm}, width= 0.33\textwidth]{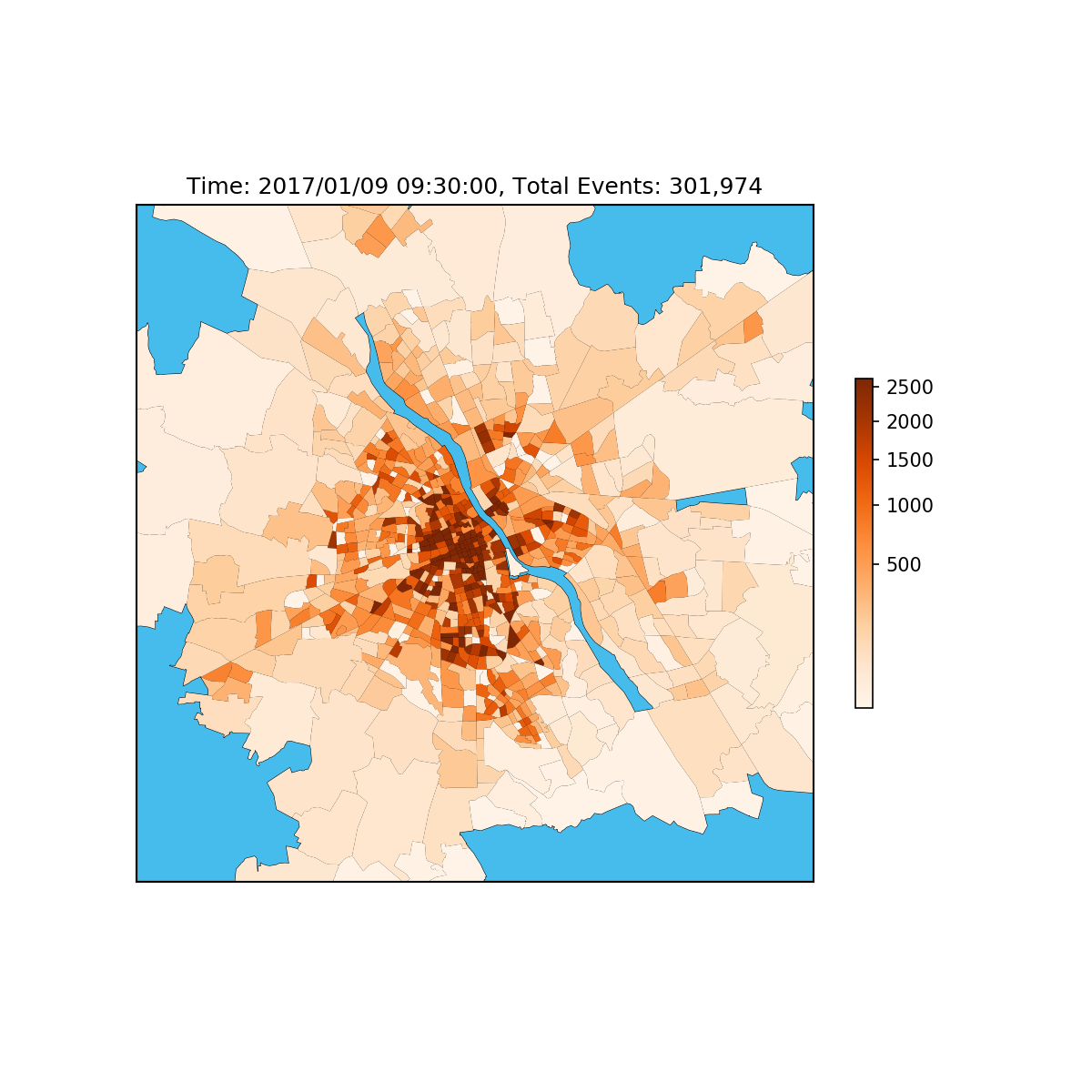}}
	\subfloat[4:30 p.\ m.\ ]{\includegraphics[clip, trim={2cm 2cm 2cm 2cm}, width= 0.33\textwidth]{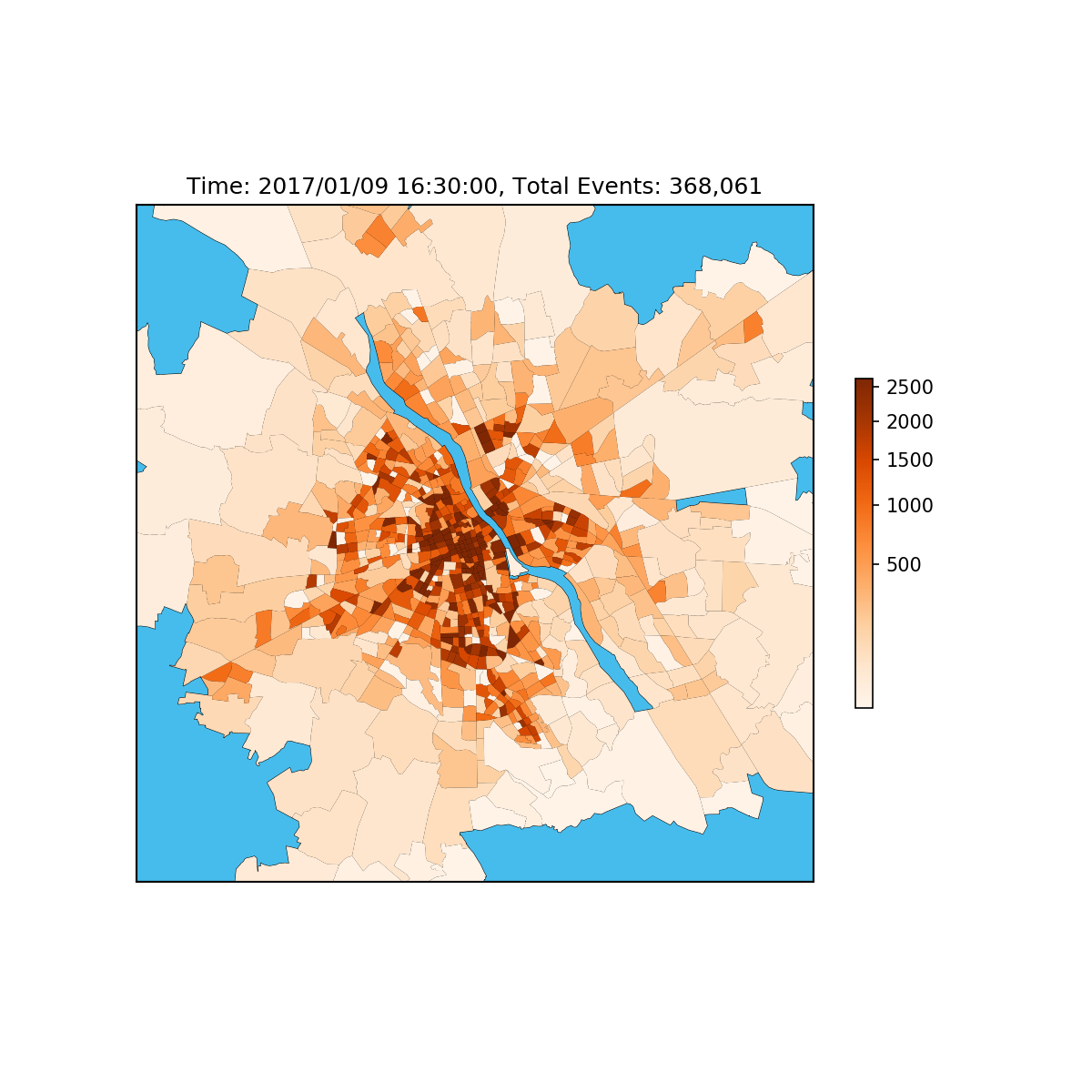}}
	\caption{Density of events across zones for different time periods. Color scale uses a power law normalization.
	}
	\label{fig:spatial}
\end{figure*}

\begin{figure}[t!]
%\subfloat[presence1]{
\includegraphics[width=\columnwidth]{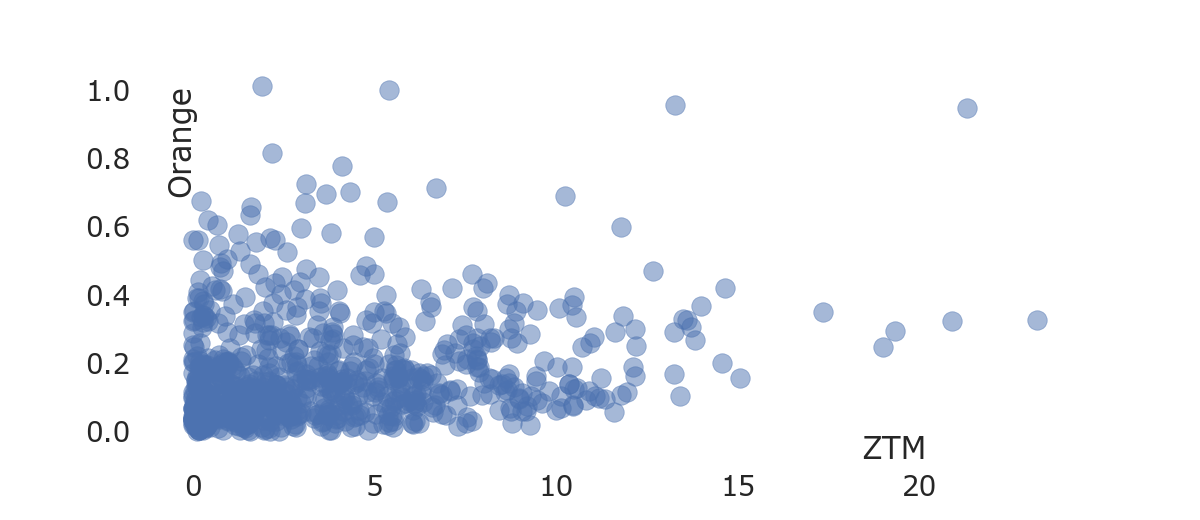}
%}
%\subfloat[presencerr]{\includegraphics[width= 0.45\textwidth]{Pix3/optimal_err.png}}
\caption{Presence by zone implied by the estimates of ZTM and the density of events in the Orange network.
}
\label{fig:presence}
\end{figure}  
 
For further comparison, we have implemented the doubly-constrained gravity model \cite{erlander1990gravity}.
There, the presence vectors $E_\ell$ across two consecutive time steps are seen as production 
and attraction vectors within the framework of a travel-demand model, which aims to generate trip distribution rates.
 An appropriate normalization can give yield matrix $X$ similar:
\begin{equation}\label{eq:gravity}
\begin{aligned}
x_{ij} = \frac{\left( A_i B_j E_{i1} E_{j2}\right)}{c_{ij}^\alpha},
\end{aligned}
\end{equation}
where $A_i$ and $B_j$ are zone-specific parameters for the origin and destination, respectively, 
 which are learned from the data, and $\alpha$ is a 
 parameter, which we assume to be $\alpha=1.0$. 
An iterative algorithm \cite{erlander1990gravity} can then be used to determine the flow, 
 subject to flow conservation constraints.

To test the approaches, data were aggregated 
 spatially to 820 zones defined by the public transport operator (ZTM),
 and temporally to 15-minute units, prior to the application 
 of any methods described in the present paper. 
Independently, for each cost function, an $820 \times 820$ matrix has
  been calculated, based on a description of each zone as a polygon with points defined by their latitude and longitude.  
Additionally, we have obtained randomised variants of the objective function by adding uniformly-distributed noise $U[0, 0.0001)$ to each
 element, e.g., $\Cadj(Z_i, Z_j)$ for all $i,j$.
(Note that the introduction of a small amount of noise makes it possible to obtain multiple optimizers 
with very similar objective function value.)
The data have been normalised to the number $N = 1,000,000$ of users, 
 i.e., such that the marginals, and hence the $672,400$ scalar variables sum up to $1,000,000$.
Subsequently, the linear program~\eqref{eq:LP1} of Section~\ref{constantnousers} with
  $672,400$ scalar variables and $1,640$ dense constraints has been formulated and solved 
  %and solved the linear program~\eqref{eq:LP1} 
  for each pair of consecutive time steps between 8 a.\ m.\ and 9.30 a.\ m.\ on each of the five work days
  and each of 4 randomisation of the objective function.
The sparse, integral solutions of all linear programs have been averaged to reduce the sparsity 
  in our estimate of the 1-step transition matrix $S_1$.
Subsequently, we have obtained $i$-step transition matrices for $i = 3, 4, \ldots 7$,
 i.e., corresponding to trips of 30, 45, 60, 75, and 90 minutes of duration, by the matrix-multiplication procedure
of Section~\ref{matrixpowers}.
Finally, we have computed a convex combination of the $i$-step transition matrices as suggested in 
Section~\ref{realistictransitions}, using weights detailed in the caption of Figure~\ref{fig:stats2}.

\begin{figure}[t!]
\subfloat[$1/6, 1/6, 1/6, 1/6, 1/6, 1/6$]{\includegraphics[width= 0.45\columnwidth]{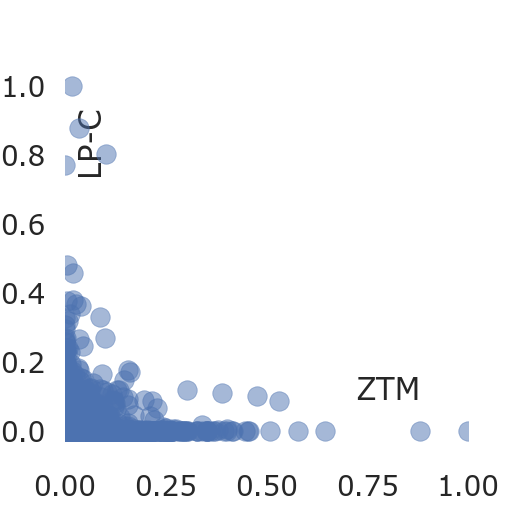}}
\subfloat[$0.5, 0.1, 0.1, 0.1, 0.1, 0.1$]{\includegraphics[width= 0.45\columnwidth]{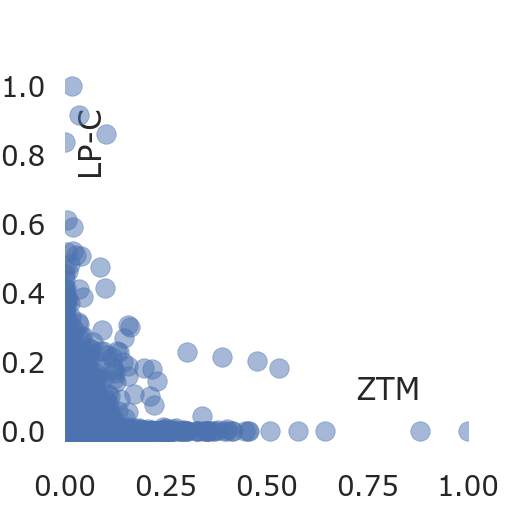}}\\
\subfloat[$1-5\epsilon, \epsilon, \epsilon, \epsilon, \epsilon, \epsilon$]{\includegraphics[width= 0.45\columnwidth]{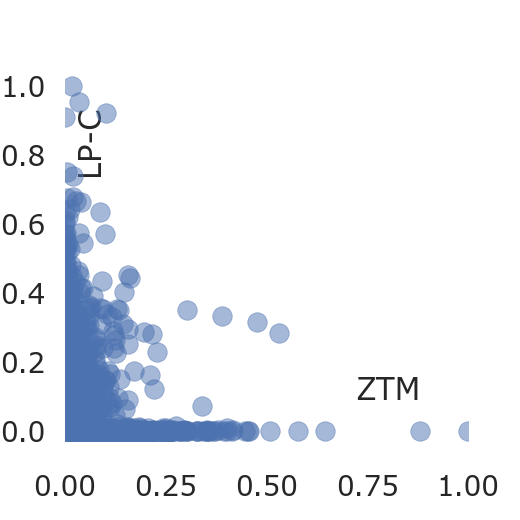}}
\subfloat[Gravity model]{\includegraphics[width= 0.45\columnwidth]{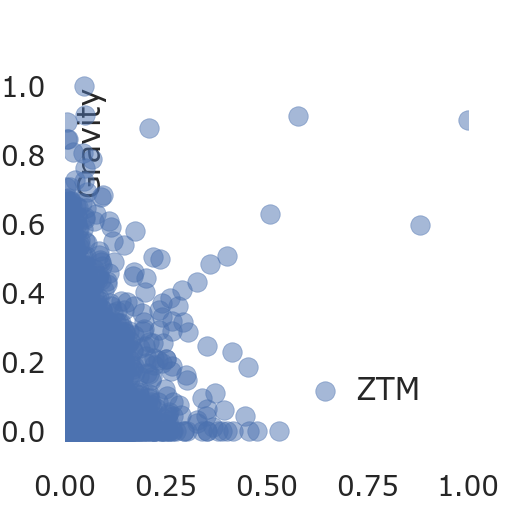}}
\caption{Solutions to Problem \ref{prob3} obtained using \eqref{eq:LP1} and 
three variants of the vector $h_1, h_2, h_3, h_4, h_5, h_6$,
compared against the estimates of the public transport operator (ZTM).}
\label{fig:stats2}
\end{figure}

Notice that we did not have an accurate estimate of the trip-duration distribution,
  which Figure~\ref{fig:stats2} suggests has a considerable impact.
Still, the match seems encouraging, considering the simplicity of the methods applied:
  we have considered the simplest cost function $\Cadj$, 
  we have assumed the constant number $N$ of users throughout in \eqref{eq:LP1}, 
  and the crudest method of arriving at the multi-step transition matrices.

\section{Related Work}

\subsection{The Demand Modelling}

There is a long history of the use of individual mobile phone subscriber data in transportation modelling. 
Outside of Europe, the most common study is based on CDR (Call Data Record). 
For example in \cite{5928310}, the CDR are used to analyse flows between 
the city and the suburban area of New York. 
In \cite{LEO201643}, researchers analyse the CDR from the main mobile 
operator in Mexico, and one of the papers focuses  on users movements analysis. 
In \cite{Wang:1502355},
CDR data including phone calls, SMS messages, and web browsing of customers of one US mobile operator in the San Francisco Bay area 
are used to determine patterns of urban road use. 
The Orange company organized two editions of the D4D Challenge (\url{http://www.d4d.orange.com/en/Accueil}), during which the anonymous CDR of 
Abidjan in Cote d'Ivoire data was made available to developers and researchers. 
Some of the present authors \cite{7117451,7471487} used 
the data set to optimize public transport. 
In a related paper \cite{GUNDLEGARD201629}, the data from the competition is 
used to estimate the travel demand and network allocation.
The access to the CDR data is particularly limited in Europe, though,
where privacy-protection regulation make access to the CDR data impossible 
without their anonymization and user's consent, in most cases.

Within Europe, \cite{ASGARI201669} discusses the use of GPS data, CDR, and cellular data (details of base stations) to determine the 
trajectory of subscriber mobility. The data for this study was collected by a group of volunteer users 
in cooperation with one of the French operators, and the study covers the metropolitan area of Paris
 (Ile-de-France). 
The use of anonymous mobile data and their use for locating subscribers anonymously 
 is described in \cite{6569081}. % and used by researchers in [10] for the Boston metropolitan area.
%The another publication [11] describes usage of 
\cite{7474365} discusses the use of data from Italy's Vodafone users to compute
 mobility patterns in Italian cities. In this case, the real-time events (due to voice, data, and SMS 
 communication) observed between A Interface and IU-CS interface in 2G and 3G mobile networks have been used. 
In \cite{4287441}, the authors describe the usage of aggregated mobile phone data from Rome, including the information
 about the base station telephone traffic, expressed in Erlang units. These records are combined with the 
 location and trajectory data of callers and data from the two city transportation companies (taxi and public bus company).
\cite{Luckner2017} describe the aggregation procedure used to derive the dataset we use. 

\subsection{The Algorithms}

Notice that the linear program \eqref{eq:LP1} has a well-known structure, in that the feasible set is the 
transportation polytope:

\begin{Def}[Transportation Polytope] \label{def:TP}
	Given $\gamma\in\reals^p$, $\eta\in\reals^q$, the set of matrices $X \in \mathbb{R}^{p \times q}, X_{ij} \geq 0$ satisfying
	\begin{equation}
	\begin{aligned}
	\textstyle \sum_{i=1}^p X_{ij} &= \eta_j   & \forall j=1,\dotsc,q\\
	\textstyle \sum_{j=1}^q X_{ij} &= \gamma_i & \forall i=1,\dotsc,p \\
	\end{aligned}\label{eq:TP}
	\end{equation}
	is called the \emph{$(p,q)$-transportation polytope defined by the 
\emph{marginals} $\gamma$ and $\eta$.} 
\end{Def}
% 	X\geq0, \quad 

Related algorithms for solving special-cases of linear programming are employed throughout pattern recognition and content-based content retrieval.
The commonly used statistical distances include the Earth mover's distance  (EMD) \cite{RubnerTomasiGuibas2000}, 
also known as the transportation cost metric \cite{Khot2006}, 
and 
Mallows distance \cite{mallows1972}, 
which are equivalent \cite{937632}, in the sense of being dual of each other.
These are discretisations of the metrics of Wasserstein and Monge-Kantorovich \cite{Deng2009},
which are also dual of each other.
Computing the distances amounts to solving a certain problem over the so-called transportation polytope.
Perhaps the best overview is provided by the two-volume 
work of Rachev and R{\" u}schendorf \cite{rachev1998mass,rachev1998mass2}.

In Computer Vision \cite{5459199,4587662,Haker2004}, 
a variety of exact and approximate methods have been considered.
Perhaps the best known exact method \cite{5459199}
uses the min-cost flow algorithm, yielding run-time of $O(n^2 \log n)$
for $n$ bins. 
Approximations based on gradient flows of certain partial differential equations
\cite{Haker2004} 
can be computed faster than earth mover's distance, empirically, but do not come 
with any guarantees, i.e., the approximation ratio.
Approximations based on the sum of absolute values of weighted wavelet coefficients of the difference of the histograms \cite{4587662} have been proposed. Such an approximation can be computed in time linear in the number of bins, but does not come with a bound on its quality, i.e., the approximation ratio.
Papers on approximations based on space-filling curves \cite{Jang2011}
have been withdrawn, recently.

In Machine Learning \cite{Cuturi2013}, Sinkhorn distances, which are a regularised variant of the Earth mover's distance, have been explored. 
Although the algorithms based on the Sinkhorn distance are substantially
faster than the network simplex for Earth mover's distance, 
It has also been realised \cite{4135678} that one can employ the following representation with $O(n)$ variables for $n$ bins:
leading to run-time cubic in the number of bins.

In Theoretical Computer Science, \cite{Khot2006} have studied the limits
of approximability of earth mover's distance, 
\cite{Andoni2008} have proposed logarithmic approximation in sub-linear time,
while \cite{Ba2010,Chan2014} proposed algorithms
approximation algorithms, whose run-time depends on the dimension of the domain space and the quality of the result required,
possibly sub-linear in the number of bins.
A similar in spirit is a recent method \cite{li2014}, which considers
dithering of the domain space, followed by max-flow computations \cite{5459199}.
We stress that our algorithms are much easier to implement and
provide the exact solution.

Overall, the present-best algorithms \cite{5459199} had run-time of $O(n^2 \log n)$ for $n$ bins, compared to $O(n)$ we present in Theorem~\ref{thm:tracemax}.
 Across all the references listed, only the value of the objective function, rather than the value of the argument at which it is achieved, is sought.
\section{Conclusions}

%The computation of Wasserstein and Monge-Kantorovich metrics 
%can be seen as trace-maximisation over infinite-dimensional
%transportation polytope. Although the infinite-dimensional
%case has been studied in some detail, c.f. Section 4.6, 
%of \cite{rachev1998mass},
%we aim to fully explicate the connections only in an extended
%version of this paper.

Across many transport-engineering applications, from planning of public transport
to balancing of a fleet of shared vehicles, one needs an up-to-date model of demand
for transportation. 
Further, one may be interested in higher moments of the demand, rather than just the commonly-used 
 expectation.
The use of mobile phone data makes it possible to address both issues.

%Mobile subscriber's location data provide a very good sample of the transportation
%patterns, but are increasingly hard to obtain by third parties, and even to process
%by the operators.
%Still, aggregate information about the per-cell activity are widely available.

For the first time, we have formalised the problem of extraction of a model of demand for transport from the aggregate data
 on the use of the base stations.
%First, we introduce the problem and study the linear programs (LP) involved.
We have shown that the associated optimisation problems have the form of a trace maximisation over the transportation polytope,
which allows for much faster algorithms than the usual simplex or interior point methods.
In particular, we show that there exist closed-form expressions for the values of the objective functions of trace maximisation over the transportation polytope. 
This allows for the computation of the objective function in time linear in the dimension of the input, in theory, 
and many orders of magnitude faster than solvers in the published literature, in practice. 
Numerous extensions are possible, 
aiming to infer mode choice,  
to study the variability of the demand over time,
as well as to improve the calibration techniques.
We envision this may open up a new direction of research on the interface of statistics and transportation engineering.

%It could also be seen as a fully polynomial-time approximation scheme (FPTAS) for the computation of certain metrics.

%It turns out that LPs have the form of a trace maximisation over the transportation polytope,
%which allows for much faster algorithms than the usual simplex or interior point methods.
%In particular, we show that there exists closed-form expressions for the values of the objective functions of trace maximisation over the transportation polytope. 
%This allows for the computation of the objective function in time linear in the dimension of the input, in theory, and many orders of magnitude faster than solvers in the published literature, in practice. 
%It could also be seen as a fully polynomial-time approximation scheme (FPTAS) for the computation of certain metrics.

\section*{Acknowledgments}
We would like to thank Marco Cuturi for having kindly provided us with his code for the comparison reported in Section~\ref{experiments}.
We would also like to thank Yossi Rubner et al. \cite{RubnerTomasiGuibas2000} and Ofir Pele and Michael Werman \cite{5459199}
for releasing their code on-line (\url{http://robotics.stanford.edu/~rubner/emd/default.htm}, 
\url{http://www.cs.huji.ac.il/~ofirpele/FastEMD/code/}). % (\texttt{emd\_hat\_gd\_metric}).
This work received funding from the European Union Horizon 2020 Programme 
(Horizon2020/2014-2020), under grant agreement no. 688380.

\bibliographystyle{plain}
\bibliography{distances,GSM,jarek}

\clearpage
\appendix

\section{Proof of Theorem~\eqref{thm:tracemax}}
\label{proof:tracemax}

Let us present the proof of

\begin{Thm*} Given $\eta, \gamma\in\integers_+^n$, with $\one^T\eta=\one^T\gamma=k$, consider
\begin{equation}\label{eq:ILP}\tag{IP2}
    z := \max_{X \in\integers_+^{n\times n}} \; \tr(X)  \;
    \text{ s.t. } X \one = \gamma, X^T \one = \eta.
\end{equation}
Then, $z = \sum_{i=1}^{n} \min\bigl\{	\eta_i,\gamma_i \bigr\}$.
This is computable in time linear in $n$.
\end{Thm*}

\begin{proof}
%The integrality of $X$ is no restriction, since for a transport problem (which we have here), integer input data guarantees an optimal integer solution, so we can consider the LP relaxation of~\eqref{eq:ILP} wlog, since the optimal value will be identical.

The proof proceeds in two steps. First, we prove that $y\leq \sum_{i=1}^{n} \min\bigl\{	\eta_i,\gamma_i \bigr\}$ for any feasible $X$. Second, we construct a feasible $X$ achieving this upper bound.

Step 1: Since $X_{ij} \geq 0$, we certainly have $X_{ij} \leq \sum_{\ell=1}^n X_{\ell j} = \eta_j$ and
$X_{ij} \leq \sum_{\ell=1}^n X_{i\ell} = \gamma_i$. Hence in particular $X_{ii} \leq  \eta_i
$ and $X_{ii} \leq  \gamma_i$, and so by choosing the tightest bound at each $i$,  $y=\sum_{i=1}^{n} X_{ii} \leq \sum_{i=1}^{n} \min\bigl\{	\eta_i,\gamma_i \bigr\}$ follows.

Step 2: Before we proceed with Step 2 proper, we make two assumptions and
explain that these are without any loss of generality.

Assumption 1: $\eta\neq\gamma$. Notice that if $\eta=\gamma$, one lets  $X_{ii}=\eta_i=\gamma_i$ for $i=1,\dotsc,n$ and $X_{ij}=0$ for all off-diagonal elements Such an $X$ trivially satisfies all constraints and $\tr(X) =  \sum_{i=1}^{n} \eta_i = \sum_{i=1}^{n} \min\bigl\{	\eta_i,\gamma_i \bigr\} =k$, and hence the result holds for this case. We can hence assume $\eta\neq\gamma$ without the loss of generality.

Assumption 2:% There exists at least one $1 \le i \le n$ such that $\eta_i > \gamma_i$, and at least one $1 \le j \le n$ such that $\eta_j < \gamma_j$, and one 
The marginals are ordered such that:
\begin{align*}
	\eta_i &\leq \gamma_i & \text{for } i=1,\dotsc,m \\
	\eta_i &> \gamma_i & \text{for } i=m+1,\dotsc,n.
\end{align*}
%The first part follows from Assumption 1.
If that is not the case, then let $P$ denote the permutation matrix so that $P\eta$ and $P\gamma$ are ordered.\footnote{%
	$P$ can be constructed easily: let $\mathcal{C}\coloneq \{i\mid \eta_i\leq\gamma_i\}$. %and $\mathcal{R}\coloneq\{i\mid \eta_i>\gamma_i\}$
	 Then $m=\card(\mathcal{C})$ and we obtain $P$ by rearranging the identity matrix $I_n$ so that the columns with indices in $\mathcal{C}$ are the first $m$ columns, in any order.  %denoting the columns of $P$ by $p_i$, we let $p_i = e_c$ for $i=1,\dotsc,m$ and $c$ ranging over the elements of $\mathcal{C}$ in any order.  )
	} 
Then, if $\bar{X}$ is feasible for~\eqref{eq:ILP} with the modified marginals $P\eta$ and $P\gamma$, then $X=P\bar{X}P$ is feasible for~\eqref{eq:ILP} and achieves the identical objective, since $\tr(P\bar{X}P)=\tr(PP\bar{X})=\tr(\bar{X})$.

Note that we certainly have $1\leq m \leq n-1$, because $m=0$ implies $\eta_i>\gamma_i$ for all $i$, which in turn implies $\sum_{i=1}^n \eta_i > \sum_{i=1}^n \gamma_i$, which contradicts  $\sum_{i=1}^n \eta_i =k= \sum_{i=1}^n \gamma_i$, hence $m\geq1$; $m=n$ together with  Assumption 1 ($\eta\neq\gamma$) leads to an analogous contradiction.

The construction of Step 2 fixes all values on the diagonal of $X$, the first $m$ columns and the last $n-m$ rows:
\begin{align*}
	X_{ii} &= \eta_i, & X_{ij} &=0 & \text{for } j=1,\dotsc,m, \quad i\neq j \\ 
	X_{ii} &= \gamma_i, & X_{ij} &=0 & \text{for } i=m+1,\dotsc,n, \quad j\neq i.	
\end{align*}
% and constraints the off-diagonal elements of $X$
Denote the remainders of the marginals by $\tld{\eta} \coloneq \eta - \min\{\eta,\gamma\}$, $\tld{\gamma} \coloneq \gamma - \min\{\eta,\gamma\}$, where the $\min\{\cdot\}$ is taken element-wise. 

\begin{figure}[t!]
	\centering\includegraphics[keepaspectratio,width=\columnwidth]{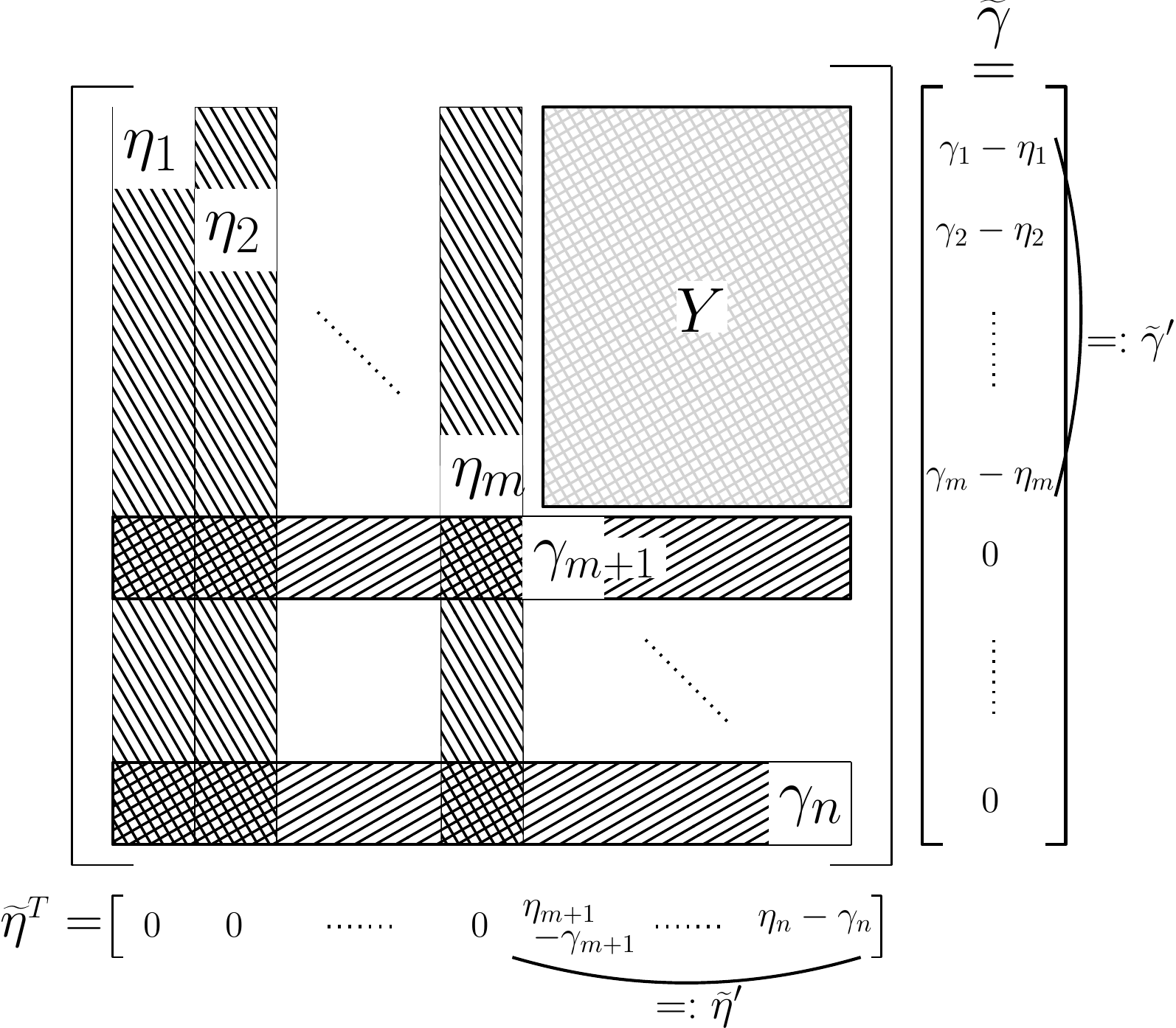}
	\caption{Illustration of the proof of Theorem~\ref{thm:tracemax}: After possibly reordering the marginals, the diagonals are filled in, and the hatched regions are filled with zeros. The elements in the rectangular block $Y$ are then used to satisfy the remaining constraints, which are collected in $\tld{\gamma}'$ and $\tld{\eta}'$.}
	\label{fig:proof}
\end{figure}

Let us now consider several properties of this construction; see also Figure~\ref{fig:proof} for an illustration:
\begin{enumerate}
	\item[P1] $\tr{X} =  \sum_{i=1}^{n} \min\bigl\{	\eta_i,\gamma_i \bigr\}$, which matches the bound of Step 1.
	\item[P2] %The off-diagonal elements within the first $m$ columns and last $n - m$ rows are uniformly 0. By Assumption 2, the constraints 
$\sum_{i=1}^n X_{ij} = \eta_j$ for $j=1,\dotsc,m$ and $\sum_{j=1}^n X_{ij} = \gamma_i$ for $i=m+1,\dotsc,n$, %are satisfied once the diagonal is set, 
i.e., all constraints but the $m$ first row sums and last $n-m$ column sums are already satisfied. 
	\item[P3] Only $X_{ij}$ for $i=1,\dotsc,m$ and $j=m+1,\dotsc,n$ are unassigned. Denote this $m\times(n-m)$ matrix by $Y$.
	\item[P4] $\one^T\tld{\eta}=\one^T\tld{\gamma}\eqcolon \tld{k}$, $\tld{\eta},\tld{\gamma}\geq0$, and $\tld{\eta}_i=\tld{\gamma}_j=0$ for $i=1,\dotsc,m$ and $j=m+1,\dotsc,n$. Denote the remaining elements of $\tld{\eta}$ and $\tld{\gamma}$ by $\tld{\eta}'$ and $\tld{\gamma}'$ respectively, i.e.\ $\tld{\eta}'\in\integers_+^{n-m}$ and $\tld{\gamma}'\in\integers_+^{m}$ and
	\[
		\tld{\eta} = \bmat{0\\ \tld{\eta}'}, \qquad \tld{\gamma} = \bmat{ \tld{\gamma}'\\0}.
	\]
	We still have $\one_{n-m}^T\tld{\eta}'=\one_m^T\tld{\gamma}'= \tld{k}$,
         since we only removed zero elements.
\end{enumerate}

By property P1 and P3, if we are able to find a matrix
$Y\in\integers_+^{m\times (n-m)}$ satisfying 
$Y\one_{n-m}=\tld{\gamma}'$, and
$Y^T\one_{m}=\tld{\eta}'$,
we can recover a feasible $X$ that achieves the bound of Step 1.
By P3 and P4, this amounts to finding a point in the ($m,n-m)$ transportation polytope defined by the marginals $\tld{\gamma}'$ and $\tld{\eta}'$.
Since by P4 we have that $\one_{n-m}^T\tld{\eta}'=\one_m^T\tld{\gamma}'$, it follows from Lemma~\ref{lem:TP} that such a $Y$ does exist, which concludes Step 2. %This step of the proof is probably much more straightforward when illustrated, see Figure~\ref{fig:proof} for that.
\end{proof}

\end{document}